\documentclass{louloupart}
\usepackage{pinlabel}
\usepackage{graphicx}
\usepackage[all]{xy}
\usepackage{tikz}
\usepackage{caption}
\usepackage[T1]{fontenc}
\title[Narrow normal subgroups of Coxeter groups]{Narrow normal subgroups of Coxeter groups and of automorphism groups of Coxeter groups}
\author[L Paris]{Luis Paris}
\givenname{Luis}
\surname{Paris}
\address{Luis Paris, IMB, UMR 5584, CNRS, Universit\'e Bourgogne, 21000 Dijon, France}
\email{lparis@u-bourgogne.fr}
\author[O Varghese]{Olga Varghese}
\givenname{Olga}
\surname{Varghese}
\address{Olga Varghese, Institute of Mathematics, Heinrich-Heine-University Düsseldorf, Universitätsstra{\upshape{\ss}}e 1, 40225, Düsseldorf, Germany}
\email{olga.varghese@hhu.de}


\newtheorem{thm}{Theorem}[section]
\newtheorem{lem}[thm]{Lemma}
\newtheorem{prop}[thm]{Proposition}
\newtheorem{corl}[thm]{Corollary}

\theoremstyle{definition}

\newtheorem*{acknow}{Acknowledgments}
\newtheorem*{expl}{Example}

\numberwithin{equation}{section}

\makeatletter
\renewcommand{\thefigure}{\ifnum \c@section>\z@ \thesection.\fi
 \@arabic\c@figure}
\@addtoreset{figure}{section}
\makeatother

\begin{document}

\def\N{\mathbb N} \def\R{\mathbb R} \def\GL{{\rm GL}}
\def\Z{\mathbb Z} \def\gen{{\rm gen}} \def\aff{{\rm aff}}
\def\sph{{\rm sph}} \def\Out{{\rm Out}} \def\K{\mathbb K}
\def\C{\mathbb C} \def\supp{{\rm supp}} \def\Pc{{\rm Pc}}
\def\Aut{{\rm Aut}} \def\Inn{{\rm Inn}} \def\SSS{\mathfrak S}
\def\Hom{{\rm Hom}} \def\Ker{{\rm Ker}} \def\Im{{\rm Im}}
\def\Q{\mathbb Q} \def\SS{\mathcal S} \def\II{\mathcal I}
\def\id{{\rm id}}


\begin{abstract}
By definition, a group is called narrow if it does not contain a copy of a non-abelian free group. 
We describe the structure of finite and narrow normal subgroups in Coxeter groups and their automorphism groups. 

\smallskip\noindent
{\bf AMS Subject Classification\ \ } 
Primary: 20F55; Secondary: 20F28

\smallskip\noindent
{\bf Keywords\ \ } 
Coxeter groups, automorphism groups of Coxeter groups, narrow normal subgroups, automatic continuity
\end{abstract}

\maketitle


\section{Introduction}\label{sec1}

A discrete group $G$ is called \emph{full-sized} if it contains a copy of the free group $F_2$ of rank $2$.
Otherwise $G$ is called \emph{narrow}.

The motivation for this paper comes from the so-called Tits alternative.
Recall that a group $G$ satis-\\fies the \emph{Tits alternative} if any narrow finitely generated subgroup of $G$ is virtually solvable.
This terminology comes from a celebrated theorem of Tits \cite{Tits1} which says that, for any field $\K$, $\GL_n(\K)$ satisfies the Tits alternative.
There are many other families of groups that satisfy the Tits alternative including hyperbolic groups (see Gromov \cite{Gromo1}), mapping class groups (see Ivanov \cite{Ivano1} and McCarthy \cite{McCar1}), $\Out(F_n)$ (see Bestvina--Feighn--Handel \cite{BeFeHa1,BeFeHa2}), groups acting on certain ``recurrent'' complexes of dimension $2$ (see Osajda--Przytycki \cite{OsaPrz1} ), and Coxeter groups (see Noskov--Vinberg \cite{NosVin1}).
In some cases the Tits alternative is more restrictive on narrow subgroups.
For example, any narrow subgroup of a hyperbolic group is virtually cyclic, and any narrow subgroup of a Coxeter group is virtually abelian (see Noskov--Vinberg \cite{NosVin1}).

In the present paper we are interested in a version of the Tits alternative for normal subgroups of Coxeter groups and of their automorphism groups.

A Coxeter group $W$ can be decomposed in the form $W=W_\gen\times W_\aff\times W_\sph$, where $W_\sph$ is a finite Coxeter group called the \emph{spherical part} of $W$, $W_\aff$ is an affine Coxeter group called the \emph{affine part} of $W$, and $W_\gen$ is an infinite non-affine Coxeter group called the \emph{generic part} of $W$. Our characterisation of narrow normal subgroups of Coxeter groups is the following.

\begin{thm}(see Theorem \ref{thm3_4})
Let $W$ be a Coxeter group and let $H$ be a normal subgroup of $W$.
\begin{itemize}
\item[(1)]
The group $H$ is finite if and only if $H$ is a subgroup of $W_\sph$.
\item[(2)]
The group $H$ is narrow if and only if $H$ is a subgroup of $W_\aff\times W_\sph$.
\end{itemize}
\end{thm}
In particular, if the spherical part of $W$ is trivial, then $W$ has no non-trivial finite normal subgroups, and if the spherical and the affine parts of $W$ are trivial, then $W$ has no non-trivial narrow normal subgroups.

We start our study of $\Aut(W)$ by determining a decomposition of $\Aut(W)$ of the form $$\Aut(W)=\Ker(\Phi)\rtimes(\Aut(W_\gen)\times\Aut(W_\aff)),$$ where $\Ker(\Phi)$ is a finite group containing $\Aut(W_\sph)$ (see Equation (\ref{eq4_1})). Then we show:

\begin{thm}(see Theorem \ref{thm4_8})
Let $W$ be a Coxeter group and let $H$ be a normal subgroup of $\Aut(W)$.
\begin{itemize}
\item[(1)]
The group $H$ is narrow if and only if $H$ is a subgroup of $\Ker(\Phi)\rtimes\Aut(W_\aff)$.
In that case $H$ is virtually $\Z^m$ for some $m\ge0$.
\item[(2)]
The group $H$ is finite if and only if $H$ is a subgroup of $\Ker(\Phi)$.
\end{itemize}
\end{thm}

In particular, if $W$ has no spherical part, then $\Aut(W)$ has no non-trivial finite normal subgroups, and if $W$ has no spherical part and no affine part, then $\Aut(W)$ has no non-trivial narrow normal subgroups.
Further any narrow normal subgroup of $\Aut(W)$ is virtually abelian of finite rank.
Note that we know from Noskov--Vinberg \cite{NosVin1} that any narrow subgroup of $W$ is virtually abelian of finite rank, without the assumption being normal, but we do not know if any narrow subgroup of $\Aut(W)$ is virtually abelian of finite rank in general.

The characterization of narrow normal subgroups of Coxeter groups and their automorphism groups has an application in the direction of automatic continuity (see Chapter \ref{sec5}).

\begin{acknow}
We want to thank Philip M\"oller for useful comments on the previous version of this paper and the referee for many helpful remarks. The first author is supported by the French project ``AlMaRe'' (ANR-19-CE40-0001-01) of the ANR. The second author is supported by DFG grant VA 1397/2-2.
\end{acknow}


\section{Preliminaries}\label{sec2}

Let $S$ be a finite set.
A \emph{Coxeter matrix} over $S$ is a square matrix $M=(m_{s,t})_{s,t\in S}$ indexed by the elements of $S$, with coefficients in $\N\cup\{\infty\}$, such that $m_{s,s}=1$ for all $s\in S$ and $m_{s,t}=m_{t,s}\ge 2$ for all $s,t\in S$, $s\neq t$.
We represent such a matrix by a labeled graph, $\Gamma$, called a \emph{Coxeter graph}, and which is defined as follows:
The set of vertices of $\Gamma$ is $S$.
Two distinct vertices $s,t\in S$ are connected by an edge if $m_{s,t}\neq 2$, and this edge is labeled by $m_{s,t}$ if $m_{s,t}\ge 4$. 

Let $\Gamma$ be a Coxeter graph and let $M=(m_{s,t})_{s,t\in S}$ be its Coxeter matrix.
The \emph{Coxeter group} $W=W[\Gamma]$ associated with $\Gamma$ is defined by the following presentation.
\[
W[\Gamma]=\langle S\mid s^2=1\text{ for all }s\in S\,,\ (st)^{m_{s,t}}=1\text{ for all } s,t\in S,\ s\neq t,\ m_{s,t}\neq\infty\rangle\,.
\]
The pair $(W,S)$ is called the \emph{Coxeter system associated} with $\Gamma$.

Let $\Gamma$ be a Coxeter graph and let $M=(m_{s,t})_{s,t\in S}$ be its Coxeter matrix.
We say that $W=W[\Gamma]$ is \emph{irreducible} if $\Gamma$ is connected.
For $X\subset S$ we denote by $W_X$ the subgroup of $W=W[\Gamma]$ generated by $X$ and by $\Gamma_X$ the full subgraph of $\Gamma$ spanned by $X$.
We know by Bourbaki \cite[p. 20]{Bourb1} that the natural homomorphism $W[\Gamma_X]\to W_X$ which sends $s$ to $s$ for all $s\in X$ is an isomorphism.
A subgroup of the form $W_X$ is called a \emph{standard parabolic subgroup} and a subgroup conjugate to $W_X$ is simply called a \emph{parabolic subgroup}.

Let $\Gamma_1,\dots,\Gamma_\ell$ be the connected components of $\Gamma$ and, for $i\in\{1,\dots,\ell\}$, let $X_i$ be the set of vertices of $\Gamma_i$.
Then 
\[
W=W_{X_1}\times W_{X_2}\times\cdots\times W_{X_\ell}\,,
\]
and each $W_{X_i}$ is an irreducible Coxeter group.
The subgroups $W_{X_i}$ are called the \emph{irreducible components} of $W$.

Let $\Gamma$ be a Coxeter graph and let $M=(m_{s,t})_{s,t\in S}$ be its Coxeter matrix.
Let $\Pi=\{\alpha_s\mid s\in S\}$ be an abstract set in one-to-one correspondence with $S$ and let $V=\bigoplus_{s\in S}\R\alpha_s$ be the real vector space having $\Pi$ as a basis.
The \emph{canonical form} of $W=W[\Gamma]$ is the symmetric bilinear form $\langle\cdot\mid\cdot\rangle$ defined on $V$ by
\[
\langle\alpha_s\mid\alpha_t\rangle=\left\{\begin{array}{ll}
-\cos(\pi/m_{s,t})&\text{if }m_{s,t}\neq\infty\,,\\
-1&\text{if }m_{s,t}=\infty\,.
\end{array}\right.
\]
To each $s\in S$ we associate a linear reflection $\rho_s:V\to V$ defined by
\[
\rho_s(v)=v-2\langle v\mid\alpha_s\rangle\,\alpha_s\,.
\]
We have a linear representation $\rho:W\to\GL(V)$ which sends $s$ to $\rho_s$ for all $s\in S$.
This representation is faithful (see Bourbaki \cite[p. 91]{Bourb1}) and it is called the \emph{canonical representation} of $W$.

Suppose that $\Gamma$ is a connected Coxeter graph, that is, $W=W[\Gamma]$ is irreducible.
Let $\langle\cdot\mid\cdot\rangle$ be the canonical form of $W$.
We say that $W$ (or $\Gamma$) is of \emph{spherical type} if $\langle\cdot\mid\cdot\rangle$ is positive definite and that $W$ (or $\Gamma$) is of \emph{affine type} if $\langle\cdot\mid\cdot\rangle$ is positive but not positive definite.
It is known that $W$ is of spherical type if and only if $W$ is finite.
On the other hand, if $W$ is of affine type, then $W$ is a semi-direct product $W=\Z^{n-1}\rtimes W_0$, where $n=|S|$ and $W_0$ is a Weyl group, which is a finite Coxeter group.
We refer to Bourbaki \cite{Bourb1}, Humphreys \cite{Humph1} or Davis \cite{Davis1} for a detailed account on these groups.
We will say that $W$ is \emph{generic} if it is neither of spherical type nor of affine type.

Suppose now that $\Gamma$ is arbitrary.
Let $W_{X_1},W_{X_2},\dots,W_{X_\ell}$ be the irreducible components of $W$.
We say that $W_{X_i}$ is a \emph{spherical component} of $W$ if $W_{X_i}$ is of spherical type, that $W_{X_i}$ is an \emph{affine component} of $W$ if $W_{X_i}$ is of affine type, and that $W_{X_i}$ is a \emph{generic component} of $W$ if $W_{X_i}$ is generic.
Without loss of generality we can assume that $W_{X_1},\dots,W_{X_p}$ are the generic components of $W$, that $W_{X_{p+1}},\dots,W_{X_q}$ are the affine components of $W$, and that $W_{X_{q+1}},\dots,W_{X_\ell}$ are the spherical components of $W$, where $0\le p\le q\le\ell$.
Then $W_\gen=W_{X_1}\times\cdots\times W_{X_p}$ is the \emph{generic part} of $W$, $W_\aff=W_{X_{p+1}}\times\dots\times W_{X_q}$ is the \emph{affine part} of $W$, and $W_\sph=W_{X_{q+1}}\times\cdots\times W_{X_\ell}$ is the \emph{spherical part} of $W$.
Note that $W=W_\gen\times W_\aff\times W_\sph$.

\begin{expl}
Let us consider the Coxeter graph $\Gamma$ in Figure \ref{fig2_1}.
The spherical part of $W[\Gamma]$ is the standard parabolic subgroup that is generated by $\left\{ s_9,s_{10}\right\}$, the affine part is the standard parabolic subgroup that is generated by $\left\{s_1, s_2, s_3, s_4, s_5\right\}$ and the generic part is generated by $\left\{s_6, s_7, s_8\right\}$.

\begin{figure}[ht!]
	\begin{center}
		\begin{tikzpicture}
			\draw[fill=black]  (0,0) circle (2pt);
			\draw[fill=black]  (2,0) circle (2pt);		
			\node at (0,-0.4) {$s_1$}; 
			\node at (2,-0.4) {$s_2$};		
			\draw (0,0)--(2,0);	
			\node at (1,0.2) {$\infty$};
			
			\node at (0,-0.4) {$s_1$}; 
			\node at (2,-0.4) {$s_2$};
			
			\draw[fill=black]  (3,0) circle (2pt);
			\node at (3,-0.4) {$s_3$}; 
			\draw[fill=black]  (5,0) circle (2pt);
			\node at (5,-0.4) {$s_4$}; 
			\draw[fill=black]  (4,1.5) circle (2pt);
			\node at (4,1.8) {$s_5$}; 
			\draw (3,0)--(5,0);
			\draw (3,0)--(4,1.5);
			\draw (5,0)--(4,1.5);
			
			\draw[fill=black]  (6,0) circle (2pt);
			\node at (6,-0.4) {$s_6$}; 
			\draw[fill=black]  (8,0) circle (2pt);
			\node at (8,-0.4) {$s_7$}; 
			\draw[fill=black]  (7,1.5) circle (2pt);
			\node at (7,1.8) {$s_8$}; 
			\draw (6,0)--(8,0);
			\draw (6,0)--(7,1.5);
			\draw (8,0)--(7,1.5);
			\node at (7,-0.2){$4$};
			\node at (6.3,1){$4$};
			\node at (7.7,1){$4$};
			
			\draw[fill=black]  (9,0) circle (2pt);
			\node at (9,-0.4) {$s_9$}; 
			\draw[fill=black]  (11,0) circle (2pt);
			\node at (11,-0.4) {$s_{10}$};
			\draw (9,0)--(11,0);

		\end{tikzpicture}
	\caption{Coxeter graph $\Gamma$}\label{fig2_1}
	\end{center}
\end{figure}
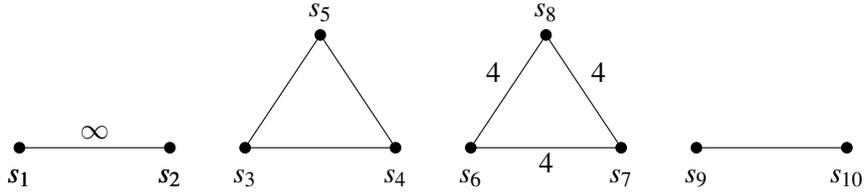
\end{expl}

Let $(W,S)$ be a Coxeter system.
The word length with respect to $S$ of an element $w\in W$ is denoted by $\ell_S(w)$.
A \emph{reduced expression} of $w$ is an expression $w=s_1s_2\cdots s_\ell$ of $w$ over $S$ such that $\ell_S(w)=\ell$.
If $w=s_1s_2\cdots s_\ell$ is a reduced expression of $w$, then the \emph{support} of $w$ is $\supp(w)=\{s_1,s_2,\dots,s_\ell\}$.
We know by \cite{Bourb1}, \cite[Prop. 4.1.1]{Davis1} that this set does not depend on the choice of the reduced expression.
Let $X,Y\subset S$ and $w\in W$.
We say that $w$ is \emph{$(X,Y)$-minimal} if it has minimal length in the double coset $W_XwW_Y$.

\begin{prop}(\cite{Bourb1}, \cite[Lemma 4.3.1]{Davis1})
\label{prop2_1}
Let $(W,S)$ be a Coxeter system, let $X,Y$ be two subsets of $S$, and let $w$ be an $(X,Y)$-minimal element of $W$.
Then:
\begin{itemize}
\item[(1)]
$w$ is the only $(X,Y)$-minimal element lying in $W_XwW_Y$,
\item[(2)]
$\ell_S(uw)=\ell_S(u)+\ell_S(w)$ for all $u\in W_X$,
\item[(3)]
$\ell_S(wv)=\ell_S(w)+\ell_S(v)$ for all $v\in W_Y$.
\end{itemize}
\end{prop}

Related to the above proposition is the following.

\begin{prop}(\cite[Ex. 22, p. 43]{Bourb1}, \cite[Lemma 4.6.1]{Davis1})
\label{prop2_2}
Let $W=W[\Gamma]$ be an irreducible Coxeter group.
The following conditions on an element $w_0\in W$ are equivalent.
\begin{itemize}
\item[(a)]
For each $u\in W$, $\ell_S(w_0)=\ell_S(u)+\ell_S(u^{-1}w_0)$,
\item[(b)]
For each $s\in S$, $\ell_S(w_0)>\ell_S(sw_0)$.
\end{itemize}
Moreover, $w_0$ exists if and only if $W$ is of spherical type.
If $w_0$ satisfies (a) and/or (b), then $w_0$ is unique, $w_0$ is an involution, and $w_0Sw_0=S$.
\end{prop}

The element $w_0$ of Proposition \ref{prop2_2} is called the \emph{longest element} of $W$, if it exists.
The center of a group $G$ will be denoted by $Z(G)$.
The following result can be found in \cite[Ch. V, \S 4, Exercise 3]{Bourb1}.

\begin{prop}\label{prop2_3}
Let $W=W[\Gamma]$ be an irreducible Coxeter group.
\begin{itemize}
\item[(1)]
If either $W$ is of affine type or $W$ is generic, then $Z(W)=\{1\}$.
\item[(2)]
Suppose that $W$ is of spherical type.
Let $w_0$ be the longest element of $W$.
Then $Z(W)=\{1,w_0\}$ if $w_0$ is central, and $Z(W)=\{1\}$ otherwise.
\end{itemize}
\end{prop}


\section{Narrow normal subgroups of a Coxeter group}\label{sec3}

In this chapter $\Gamma$ denotes a Coxeter graph, $M=(m_{s,t})_{s,t\in S}$ denotes its Coxeter matrix and $W=W[\Gamma]$ denotes its Coxeter group.
Our aim is to determine the finite normal subgroups and the narrow normal subgroups of $W$.

We know by Qi \cite{Qi1} that the intersection of any family of parabolic subgroups of $W$ is still a parabolic subgroup.
Hence we can define the \emph{parabolic closure} of a subset $A\subset W$, denoted $\Pc(A)$, as the intersection of all parabolic subgroups containing $A$.
Note that, if $H$ is a normal subgroup of $W$, then $\Pc(H)$ is also a normal subgroup.
This observation combined with the following lemma are the key points of the present chapter. 

\begin{lem}\label{lem3_1}
Suppose $\Gamma$ is connected.
Let $P$ be a parabolic subgroup of $W$.
If $P$ is normal, then either $P=\{1\}$ or $P=W$.
\end{lem}

\begin{proof}
Let $P$ be a normal parabolic subgroup of $W$.
We assume that $P\neq\{1\}$ and $P\neq W$, and we show that there is a contradiction.
We can assume without loss of generality that $P=W_X$ for some $X\subset S$ with $X\neq\emptyset$ and $X\neq S$.
Since $\Gamma$ is connected, we can find $s\in X$ and $t\in S\setminus X$ such that $m_{s,t}\neq 2$.
It is clear that $t$ is $(X,X)$-minimal.
Moreover, since $W_X$ is normal, there is $v\in W_X$ such that $tst=v$.
By Proposition \ref{prop2_1}, 
\[
\ell_S(v)+1=\ell_S(tv)=\ell_S(st)=2\,,
\]
hence $\ell_S(v)=1$, that is, there exists $s'\in X$ such that $v=s'$.
As $st=ts'$, we have $\supp(ts) =\{t,s\}=\{t,s'\}$, hence $s=s'$, and therefore $st=ts$.
This last equality contradicts the hypothesis $m_{s,t}\neq 2$.
\end{proof}

We start our study with the case where $\Gamma$ is connected, that is, where $W$ is irreducible.
If $W$ is of spherical type, then $W$ is finite and all normal subgroups of $W$ are narrow and even finite.
Now we consider the case where $W$ is of affine type.  

\begin{prop}(see \cite[Cor. 0.3]{Maxwell1998})\label{prop3_2}
Suppose that $W$ is an irreducible Coxeter group of affine type.
Any non-trivial normal subgroup of $W$ is virtually $\Z^{n-1}$, where $n=|S|$.
In particular such a subgroup is narrow, but it is not finite.
\end{prop}

\begin{proof}
Recall that $W$ can be decomposed as a semi-direct product $W=\Z^{n-1}\rtimes W_0$, where $n=|S|$, $W_0$ is a Weyl group and the action of $W_0$ on $\Z^{n-1}$ is irreducible.
Here by irreducible we mean that the linear representation $\rho:W_0\to\GL_{n-1}(\Z)\subset\GL_{n-1}(\C)$ associated with the semi-direct product decomposition is irreducible.
Let $H$ be a non-trivial normal subgroup of $W$.
Since $\Pc(H)$ is also a normal subgroup we have $\Pc(H)=W$ by Lemma \ref{lem3_1}.
If $H$ were finite, then, by Bourbaki \cite[Ch. V, \S 4, Exercise 2\,(d)]{Bourb1} (see also Tits \cite{Tits2}), $\Pc(H)$ would be finite, which is not possible because $\Pc(H)=W$ is infinite.
So, $H$ is infinite, hence $H\cap\Z^{n-1}\neq\{0\}$.
Since the representation $\rho$ is irreducible, we deduce that $H\cap\Z^{n-1}$ is a rank $n-1$ lattice, hence $H$ is virtually $\Z^{n-1}$.
\end{proof}

In the case where $W$ is a generic irreducible Coxeter group we have the following.

\begin{prop}\label{prop3_3}
Let $W$ be a generic irreducible Coxeter group.
Then $W$ does not contain any narrow normal subgroup other than the trivial group.
\end{prop}

\begin{proof}
We assume that $W$ contains a non-trivial narrow normal subgroup, $H$, and we show that there is a contradiction.
As in the proof of Proposition \ref{prop3_2} we have $\Pc(H)=W$.
In particular $H$ cannot be finite otherwise, by Bourbaki \cite[Ch. V, \S 4, Exercise 2\,(d)]{Bourb1}, $\Pc(H)=W$ would be finite.
Moreover, by Noskov--Vinberg \cite{NosVin1}, $H$ is virtually a finite rank abelian group, hence $H$ contains a finite index subgroup isomorphic to $\Z^m$ for some $m\ge1$.
So, without loss of generality we can assume that $H\simeq\Z^m$ for some $m\ge1$.
Since $\Pc(H)=W$, by \cite[Theorem 6.8.2]{Kramm1}, $m=1$. 
Since $H$ is normal, $W$ acts on $H$ by conjugation, thus we obtain a homomorphism from $W\to\Aut(H)$ whose kernel is the centralizer of $H$ in $W$. 
The automorphism group of $H\cong\Z$ is a cyclic group of order $2$, hence the centralizer of $H$ in $W$ has at most index $2$. 
On the other hand $H$ is of finite index in its centralizer by \cite[Corollary 6.3.10]{Kramm1}. 
This shows that $H$ is of finite index in $W$. 
Hence $W$ has a normal infinite cyclic subgroup of finite index. 
By \cite[Lemma 5.2]{Qi} we know that $W$ is isomorphic to the infinite dihedral group which is an affine Coxeter group, a contradiction.
\end{proof}

The case of an arbitrary Coxeter group follows from Propositions \ref{prop3_2} and \ref{prop3_3}.

\begin{thm}\label{thm3_4}
Let $W$ be a Coxeter group and let $H$ be a normal subgroup of $W$.
\begin{itemize}
\item[(1)]
The group $H$ is finite if and only if $H$ is a subgroup of $W_\sph$.
\item[(2)]
The group $H$ is narrow if and only if $H$ is a subgroup of $W_\aff\times W_\sph$.
\end{itemize}
\end{thm}

\begin{proof}
Let $W_{X_1},\dots,W_{X_\ell}$ be the irreducible components of $W$.
We assume that $W_{X_1},\dots,W_{X_p}$ are the generic components of $W$, that $W_{X_{p+1}},\dots,W_{X_q}$ are the affine components of $W$, and that $W_{ X_{q+1}},\dots,W_{X_\ell}$ are the spherical components of $W$, where $0\le p\le q\le\ell$.
In particular, $W_\gen=W_{X_1}\times\cdots\times W_{X_p}$, $W_\aff=W_{X_{p+1}}\times\cdots\times W_{X_q}$ and $W_\sph=W_{X_{q+1}}\times\cdots\times W_{X_\ell}$.
Recall that 
\[
W=W_{X_1}\times\cdots\times W_{X_\ell}=W_\gen\times W_\aff\times W_\sph\,.
\]
For $i\in\{1,\dots,\ell\}$ we denote by $\pi_i:W\to W_{X_i}$ the projection on the $i$-th component.

Suppose $H$ is narrow.
Then, for $i\in\{1,\dots,p\}$, $\pi_i(H)$ is a narrow normal subgroup of $W_{X_i}$, so, by Proposition \ref{prop3_3}, $\pi_i(H)=\{1\}$.
This implies that $H$ is a subgroup of $W_\aff\times W_\sph$.
Conversely, it is easily seen that, if $H$ is a subgroup of $W_\aff\times W_\sph$, then $H$ is narrow.

Suppose $H$ is finite.
Then, for $i\in\{1,\dots,p,p+1,\dots,q\}$, $\pi_i(H)$ is a finite normal subgroup of $W_{X_i}$, so, by Propositions \ref{prop3_2} and \ref{prop3_3}, $\pi_i(H)=\{1\}$.
This implies that $H$ is a subgroup of $W_\sph$.
Conversely, it is easily seen that, if $H$ is a subgroup of $W_\sph$, then $H$ is finite.
\end{proof}

\begin{corl}\label{corl3_5}
Let $W$ be a Coxeter group.
\begin{itemize}
\item[(1)]
If $W_\sph=\{1\}$, then $W$ does not contain any non-trivial finite normal subgroup.
\item[(2)]
If $W_\sph=W_\aff=\{1\}$, then $W$ does not contain any non-trivial narrow normal subgroup.
\end{itemize}
\end{corl}

At the end of this chapter we want to point out that the property of a group not having non-trivial finite normal subgroups is inherited to finite index normal subgroups.

\begin{prop}
	Let $G$ be a group and $N\subseteq G$ be a finite index normal subgroup. If $G$ does not have non-trivial finite normal subgroups, then also $N$ does not have non-trivial finite normal subgroups.  
\end{prop}

\begin{proof} 
First we recall a special case of Dietzmann's Lemma (see \cite[\S 2.1 Cor. 2]{Robinson}): Let $G$ be a group and $x\in G$ be an element with finite order. If the set $X:=\left\{gxg^{-1}|g\in G\right\}$ is finite, then the subgroup generated by $X$ is a finite normal subgroup in $G$.

Let $F$ be a finite normal subgroup in $N$. For $f\in F$, the set $\left\{nfn^{-1}|n\in N\right\}$ is finite since $F$ is normal in $N$. The subgroup $N$ is normal in $G$ and has finite index, thus the set $\left\{gfg^{-1}|g\in G\right\}$ is finite.	By Dietzmann's Lemma we know that the subgroup generated by $\left\{gfg^{-1}\mid g\in G\right\}$ is a normal finite subgroup in $G$. By assumption, $G$ does not have non-trivial normal subgroups, hence $f=1$ and therefore $F$ is trivial.  
\end{proof}

Our last observation is the following.
\begin{prop}
Let $G$ be a group and $N\subseteq G$ be a normal subgroup. If $G$ does not have non-trivial narrow normal subgroups, then $Z(N)$ is trivial. 	
\end{prop}
\begin{proof}
Since $N$ is normal in $G$, the group $G$ acts on $N$ via conjugation and we obtain a group homorphism $G\rightarrow \Aut(N)$ whose kernel is the centralizer of $N$ in $G$ which we denote by $C_G(N):=\left\{g\in G| gn=ng\text{ for all }n\in N\right\}$. The intersection of normal subgroups is always normal, thus $N\cap C_G(N)$ is normal in $G$. It is easy to verify that $N\cap C_G(N)=Z(N)$, hence $Z(N)$ is an abelian normal subgroup in $G$, but the only abelian normal subgroup in $G$ is the trivial one, thus $Z(N)=\left\{1\right\}$.  	
\end{proof}


\section{Narrow normal subgroups of the automorphism group of a Coxeter group}\label{sec4}

Our aim now is to determine the finite normal subgroups and the narrow normal subgroups of the automorphism  group of a Coxeter group.
As in the previous chapter, we start with the case where $W$ is irreducible.
The case where $W$ is finite is trivial, since in this case $\Aut(W)$ is finite, hence any (normal) subgroup of $\Aut(W)$ is finite.

Let $\Gamma$ be a Coxeter graph and let $W=W[\Gamma]$ be its Coxeter group.
The conjugation by an element $w\in W$ will be denoted $\gamma_w:W\to W$, $u\mapsto wuw^{-1}$.
We have a homomorphism $\gamma:W\to\Aut(W)$, which sends $w$ to $\gamma_w$ for all $w\in W$, whose image, denoted $\Inn(W)$, is formed by the \emph{inner automorphisms}.
The kernel of $\gamma$ is the center of $W$, denoted $Z(W)$.
By Proposition \ref{prop2_3}, if $W$ is irreducible of affine type or generic, then $Z(W)=\{1\}$, hence $\gamma$ is injective.
On the other hand, the automorphism group of $\Gamma$, denoted $\Aut(\Gamma)$, can and will be considered as a subgroup of $\Aut(W)$.

The case where $W$ is an irreducible Coxeter group of affine type follows from the following result combined with Proposition \ref{prop3_2}.

\begin{thm}[Franzsen--Howlett \cite{FraHow1}]\label{thm4_1}
Let $\Gamma$ be a connected Coxeter graph of affine type and let $W=W[\Gamma]$ be its Coxeter group.
Then
\[
\Aut(W)=\Inn(W)\rtimes\Aut(\Gamma)\simeq W\rtimes\Aut(\Gamma)\,.
\]
\end{thm}

\begin{prop}\label{prop4_2}
Let $\Gamma$ be a connected Coxeter graph of affine type and let $W=W[\Gamma]$ be its Coxeter group.
Any non-trivial normal subgroup of $\Aut(W)$ is virtually $\Z^{n-1}$, where $n=|S|$.
In particular such a subgroup is narrow and infinite.
\end{prop}

\begin{proof}
Let $H$ be a normal subgroup of $\Aut(W)$.
We claim that, if $H\cap\Inn(W)=\{\id\}$, then $H=\{\id\}$.
Suppose $H\cap\Inn(W)=\{\id\}$.
Let $f\in H$.
By Theorem \ref{thm4_1}, $f$ can be written in the form $f=\gamma_ug$ with $u\in W$ and $g\in\Aut(\Gamma)$.
Let $s\in S$.
Since $H$ is normal, we have $\gamma_sf^{-1}\gamma_s^{-1}\in H$, hence $f(\gamma_sf^{-1}\gamma_s^{-1})\in H$, and therefore 
\[
f(\gamma_sf^{-1}\gamma_s^{-1})=\gamma_ug\gamma_sg^{-1}\gamma_u^{-1}\gamma_s^{-1}=\gamma_u\gamma_{g(s)}
\gamma_u^{-1}\gamma_s^{-1}=\gamma_{u\,g(s)\,u^{-1}s^{-1}}\in H\cap\Inn(W)=\{\id\}\,.
\]
It follows that $u\,g(s)\,u^{-1}s^{-1}=1$, hence $g(s)=u^{-1}su$.
This equality is true for all $s\in S$, hence $g=\gamma_{u^{-1}}$, and therefore $f=\gamma_u\gamma_{u^{-1}}=\id$.
So, $H=\{\id\}$.

Assume $H$ is a non-trivial normal subgroup of $\Aut(W)$.
Then, by the above, $H\cap\Inn(W)$ is a normal non-trivial subgroup of $\Inn(W)\simeq W$, hence, by Proposition \ref{prop3_2}, $H\cap\Inn(W)$ is virtually $\Z^{n-1}$.
Since $\Aut(\Gamma)$ is finite, $H\cap\Inn(W)$ is a finite index subgroup of $H$, hence $H$ is also virtually $\Z^{n-1}$.
\end{proof}

Concerning the case where $W$ is a generic irreducible Coxeter group we have the following result which also follows from M\"oller--Varghese \cite[Lemma 2.14]{MoellerVarghese}.

\begin{prop}\label{prop4_3}
Let $W$ be a generic irreducible Coxeter group.
Then $\Aut(W)$ contains no narrow normal subgroup other than the trivial group.
\end{prop}

\begin{proof}
Let $H$ be a narrow normal subgroup of $\Aut(W)$.
The group $H\cap\Inn(W)$ is a normal narrow subgroup of $\Inn(W)\simeq W$, hence, by Proposition \ref{prop3_3}, $H\cap\Inn(W)=\{\id\}$.
Let $h\in H$.
For each $w\in W$ we have $\gamma_wh^{-1}\gamma_w^{-1}\in H$, since $H$ is normal, hence 
\[
h(\gamma_wh^{-1}\gamma_w^{-1})=(h\gamma_wh^{-1})\gamma_w^{-1}=\gamma_{h(w)}\gamma_w^{-1}\in H\cap\Inn(W)=
\{\id\}\,,
\]
thus $\gamma_{h(w)}=\gamma_w$, and therefore $h(w)=w$.
So, $h=\id$.
Thus, $H=\{\id\}$.
\end{proof}

Let $W$ be a Coxeter group.
Recall that $Z(W)$ denotes the center of $W$.
An automorphism $f\in\Aut(W)$ is called \emph{central} if $f(w)\,w^{-1}\in Z(W)$ for all $w\in W$.
To deal with the general case we will need the following result which is a direct consequence of Theorems 4.1 and 5.2 in  \cite{Paris1}.

\begin{thm}[Paris \cite{Paris1}]\label{thm4_4}
Let $W$ be a Coxeter group.
Let $W_{X_1},\dots,W_{X_p}$ be the generic components of $W$ and let $W_{X_{p+1}},\dots,W_{X_q}$ be the affine components of $W$.
Let $f\in\Aut(W)$.
There exist permutations $\sigma\in\SSS_p$ and $\tau\in\SSS_{q-p}$ and a central automorphism $\varphi\in\Aut(W)$ such that $f(W_{X_i})= \varphi(W_{X_{\sigma(i)}})$ for all $i\in\{1,\dots,p\}$, $f(W_{X_{p+j}})=\varphi(W_{X_{p+\tau(j)}})$ for all $j\in\{1,\dots,q-p\}$ and $f(W_\sph)=W_\sph$.
\end{thm}

It is easily seen from Theorem \ref{thm4_4} that, if $Z(W)=\{1\}$, then
\[
\Aut(W)=\Aut(W_\gen)\times\Aut(W_\aff)\times\Aut(W_\sph)\,.
\]
The next step in our analysis consists of extending this decomposition to all cases (see Corollary \ref{corl4_7}).

Let $W$ be a Coxeter group.
We denote by $\hat\iota:W_\gen\times W_\aff\to W= W_\gen\times W_\aff\times W_\sph$ the embedding which sends $W_\gen\times W_\aff$ into the first two components, and we denote by $\hat\pi:W=W_\gen\times W_\aff\times W_\sph\to W_\gen\times W_\aff$ the projection onto the first two components.
For $f\in \Aut(W)$ we set
\[
\Phi(f)=\hat\pi\circ f\circ\hat\iota:W_\gen\times W_\aff\to W_\gen\times W_\aff\,.
\]

\begin{lem}\label{lem4_5}
Let $W$ be a Coxeter group.
\begin{itemize}
\item[(1)]
We have $\Phi(f)\in\Aut(W_\gen\times W_\aff)$ for all $f\in\Aut(W)$.
\item[(2)]
The map $\Phi:\Aut(W)\to\Aut(W_\gen\times W_\aff)$ is a group homomorphism.
\end{itemize}
\end{lem}

\begin{proof}
Let $f\in\Aut(W)$.
We first show that, for all $w\in W_\gen\times W_\aff$, there exists $\zeta(w)\in Z(W)$ such that $f(w)=\Phi(f)(w)\,\zeta(w)$.
By Theorem \ref{thm4_4} we know that there exists a central automorphism $\varphi$ such that $f(W_\gen\times W_\aff) =\varphi(W_\gen\times W_\aff)$.
We take $w'\in W_\gen\times W_\aff$ such that $f(w)=\varphi(w')$ and we set $\zeta(w)=\varphi(w')\,w '^{-1}\in Z(W)$.
By Proposition \ref{prop2_3} we have $\zeta(w)\in W_\sph$, hence $(\hat\pi\circ\varphi)(w')=w'$.
Thus,
\[
f(w)=\varphi(w')=w'\,\zeta(w)=(\hat\pi\circ\varphi)(w')\,\zeta(w)=(\hat\pi\circ f)(w)\,\zeta(w)=
\Phi(f)(w)\,\zeta(w)\,.
\]

Now we show that $\Phi(f)$ is surjective.
Let $w'\in W_\gen\times W_\aff$.
There exists $w\in W_\gen\times W_\aff$ such that $f(w)=\varphi(w')$.
Since $\varphi(w')w'^{-1}=\zeta(w)\in Z(W)\subset W_\sph$, projecting onto $W_\gen\times W_\aff$ we get $ \Phi(f)(w)=w'$.
So, $\Phi(f)$ is surjective.

Now we show that $\Phi(f)$ is injective.
Let $w\in\Ker(\Phi(f))$.
Then
\[
f(w)=\Phi(f)(w)\,\zeta(w)=\zeta(w)\in Z(w)\,.
\]
Since $f$ is an automorphism, this equality implies that $w$ is central in $W$.
But, by Proposition \ref{prop2_3}, $Z(W)\cap(W_\gen\times W_\aff)=\{1\}$, hence $w=1$.
This shows that $\Phi(f)$ is injective and thus ends the proof of the first part of the lemma.

Now we show the second part of the lemma.
Let $f_1,f_2\in\Aut(W)$.
From the above we know that, for each $w\in W_\gen\times W_\aff$, there exist $\zeta(w),\zeta_1(w),\zeta_2(w)\in Z(W)$ such that
\[
(f_1f_2)(w)=\Phi(f_1f_2)(w)\,\zeta(w)\,,\ f_1(w)=\Phi(f_1)(w)\,\zeta_1(w)\,,\ f_2(w)=\Phi(f_2)(w)\,
\zeta_2(w)\,.
\]
Let $w\in W_\gen\times W_\aff$.
Then
\begin{gather*}
\Phi(f_1f_2)(w)=(f_1f_2)(w)\,\zeta(w)^{-1}=f_1\big(\Phi(f_2)(w)\,\zeta_2(w)\big)\,\zeta(w)^{-1}=\\
f_1(\Phi(f_2)(w))\,f_1(\zeta_2(w))\,\zeta(w)^{-1}=(\Phi(f_1)\circ\Phi(f_2))(w)\, \zeta_1(\Phi(f_2)(w))
\,f_1(\zeta_2(w))\,\zeta(w)^{-1}\,.
\end{gather*}
We have $\zeta_1(\Phi(f_2)(w)),\zeta(w)\in Z(W)\subset W_\sph$ by definition and $f_1(\zeta_2(w))\in Z(W)\subset W_\sph$, since $f_1$ is an automorphism and $Z(W)$ is a characteristic subgroup.
Thus, projecting onto $W_\gen\times W_\aff$ we get
\[
\Phi(f_1f_2)(w)=(\Phi(f_1)\circ\Phi(f_2))(w)\,.
\]
So, $\Phi$ is a homomorphism.
\end{proof}

By Theorem \ref{thm4_4} we have
\[
\Aut(W_\gen\times W_\aff)=\Aut(W_\gen)\times\Aut(W_\aff)\,.
\]
Moreover, the natural embedding of $\Aut(W_\gen\times W_\aff)$ into $\Aut(W)=\Aut(W_\gen\times W_\aff\times W_\sph)$ is a section of $ \Phi$, hence $\Phi$ is surjective and we have the decomposition
\begin{equation}\label{eq4_1}
\Aut(W)=\Ker(\Phi)\rtimes\big(\Aut(W_\gen)\times\Aut(W_\aff)\big)\,.
\end{equation}

Again by Theorem \ref{thm4_4} we have $f(W_\sph)=W_\sph$ for all $f\in\Ker(\Phi)\subset\Aut(W)$.
So there is a homomorphism $\Psi:\Ker(\Phi)\to\Aut(W_\sph)$ which maps $f\in\Ker(\Phi)$ to $f|_{W_\sph}$.
The inclusion map from $\Aut(W_\sph)$ into $\Ker(\Phi)$ is a section of $\Psi$, hence $\Psi$ is surjective and we have the decomposition
\begin{equation}\label{eq4_2}
\Ker(\Phi)=\Ker(\Psi)\rtimes\Aut(W_\sph)\,.
\end{equation}

If $G$ is a group and $A$ is an abelian group, then $\Hom(G,A)$ is naturally endowed with an abelian group structure, where the product $f_1\cdot f_2$ of two elements $f_1,f_2\in\Hom(G,A)$ is defined by
\[
(f_1\cdot f_2)(w)=f_1(w)\,f_2(w)\,,\ w\in G\,.
\]
We denote by $\Hom_I(G,Z(G))$ the subgroup of $\Hom(G,Z(G))$ formed of the elements of $\Hom(G, Z(G))$ that  are the identity on $Z(G)$.
Then we have an embedding $\iota:\Hom_I(G,Z(G))\hookrightarrow\Aut(G)$ defined by
\[
\iota(f)(w)=w\,f(w)\,,\ f\in\Hom_I(G,Z(G))\,, w\in G\,.
\]
Note that the image of $\iota$ is formed by central automorphisms, but not by all central automorphisms in general.

Let $W$ be a Coxeter group.
Since $Z(W)\subset W_\sph$, we can consider $\Hom(W_\gen\times W_\aff,Z(G))$ as a subgroup of $\Hom_I(W,Z(W ))$.
Then we denote by $\Upsilon:\Hom(W_\gen\times W_\aff,Z(W))\to\Aut(W)$ the restriction of $\iota$ to $\Hom(W_\gen\times W_\aff,Z(W))$.
Note that $\Upsilon$ is injective and $\Im(\Upsilon)\subset\Ker(\Psi)$.

\begin{lem}\label{lem4_6}
Let $W$ be a Coxeter group.
\begin{itemize}
\item[(1)]
We have $\Im(\Upsilon)=\Ker(\Psi)$.
\item[(2)]
The groups $\Ker(\Psi)$ and $\Ker(\Phi)$ are finite.
\end{itemize}
\end{lem}

\begin{proof}
The inclusion $\Im(\Upsilon)\subset\Ker(\Psi)$ is obvious, so we only need to show $\Ker(\Psi)\subset\Im(\Upsilon)$.
Let $f\in\Ker(\Psi)$.
By Theorem \ref{thm4_4} there exists a central automorphism $\varphi\in\Aut(W)$ such that $f(W_\gen)=\varphi(W_\gen)$, $f(W_\aff)=\varphi( W_\aff)$ and $f(W_\sph)=W_\sph$.
Since $f\in\Ker(\Psi)$, we necessarily have $f=\varphi$ and $\varphi$ is the identity on $W_\sph$.
Let $h:W_\gen\times W_\aff\to Z(W)$ be the map defined by $h(w)=\varphi(w)\,w^{-1}$.
For all $w_1,w_2\in W_\gen\times W_\aff$,
\begin{gather*}
h(w_1w_2)=\varphi(w_1w_2)\,w_2^{-1}w_1^{-1}=\varphi(w_1)\,\varphi(w_2)\,w_2^{-1}w_1^{-1}=\varphi(w_1)
\,h(w_2)\,w_1^{-1} =\\
\varphi(w_1)\,w_1^{-1}\,h(w_2)=h(w_1)\,h(w_2)\,,
\end{gather*}
hence $h\in\Hom(W_\gen\times W_\aff,Z(W))$.
It is easily seen that $f=\varphi=\Upsilon(h)$.
So, $\Ker(\Psi)\subset\Im(\Upsilon)$.

The group $\Hom(W_\gen\times W_\aff,Z(W))$ is finite since $W_\gen\times W_\aff$ is finitely generated and $Z(W)$ is finite.
Since $\Ker(\Psi)$ is isomorphic to $\Hom(W_\gen\times W_\aff,Z(G))$, it follows that $\Ker(\Psi)$ is finite.
Finally, Equation (\ref{eq4_2}) implies that $\Ker(\Phi)$ is also finite.
\end{proof}

The following corollary is easily obtained by combining Equations (\ref{eq4_1}) and (\ref{eq4_2}) with Lemma \ref{lem4_6}.

\begin{corl}\label{corl4_7}
Let $W$ be a Coxeter group.
Then
\[
\Aut(W)\simeq\big(\Hom(W_\gen\times W_\aff,Z(W))\rtimes\Aut(W_\sph)\big)\rtimes\big(\Aut(W_\gen)\times
\Aut(W_\aff)\big)\,.
\]
\end{corl}

We turn now to prove the main result of the chapter.

\begin{thm}\label{thm4_8}
Let $W$ be a Coxeter group and let $H$ be a normal subgroup of $\Aut(W)$.
\begin{itemize}
\item[(1)]
The group $H$ is narrow if and only if $H$ is a subgroup of $\Ker(\Phi)\rtimes\Aut(W_\aff)$.
In that case $H$ is virtually $\Z^m$ for some $m\ge0$.
\item[(2)]
The group $H$ is finite if and only if $H$ is a subgroup of $\Ker(\Phi)$.
\end{itemize}
\end{thm}

\begin{proof}
We first prove that there exists $n\ge0$ such that $\Aut(W_\aff)$ is virtually $\Z^n$.
Let $W_{X_{p+1}},\dots,W_{X_q}$ be the affine components of $W$.
By Theorem \ref{thm4_4} there exists a homomorphism $\alpha:\Aut(W_\aff)\to\SSS_{q-p}$ such that, for all $f\in \Aut(W_\aff)$ and all $j\in\{1,\dots,q-p\}$, we have $f(W_{X_{p+j}})=W_{X_{p+\alpha(f)(j)}}$.
Note that
\[
\Ker(\alpha)=\Aut(W_{X_{p+1}})\times\cdots\times\Aut(W_{X_q})\,.
\]
By Proposition \ref{prop4_2} each $\Aut(W_{X_{p+j}})$ is virtually $\Z^{|X_{p+j}|-1}$, so $\Ker(\alpha)$ is virtually $\Z^n$, where $n=\left(\sum_{j=1}^{q-p}|X_{p+j}|\right)-q+p$.
Since $\Ker(\alpha)$ is a finite index subgroup of $\Aut(W_\aff)$, we conclude that $\Aut(W_\aff)$ is also virtually $\Z^n$.

By Lemma \ref{lem4_6} $\Aut(W_\aff)$ has finite index in $\Ker(\Phi)\rtimes\Aut(W_\aff)$, hence $\Ker(\Phi)\rtimes\Aut(W_\aff)$ is virtually $\Z^n$, and therefore any subgroup of $\Ker(\Phi)\rtimes\Aut(W_\aff)$ is virtually $\Z^m$ for some $m\ge0$.

Let $H$ be a normal subgroup of $\Aut(W)$.
Suppose $H$ is a subgroup of $\Ker(\Phi)\rtimes\Aut(W_\aff)$.
Then, by the above, $H$ is virtually $\Z^m$ for some $m\ge0$.
In particular, $H$ is narrow.

Suppose $H$ is narrow.
We denote by $\pi_\gen:\Aut(W_\gen)\times\Aut(W_\aff)\to\Aut(W_\gen)$ the projection onto the first component and we consider the composition $\pi_\gen\circ\Phi:\Aut(W)\to\Aut(W_\gen)$.
The kernel of $\pi_\gen\circ\Phi$ is $\Ker(\Phi)\rtimes\Aut(W_\aff)$, hence, in order to show that $H\subset\Ker(\Phi)\rtimes\Aut(W_\aff)$, it suffices to show that $(\pi_\gen\circ\Phi)(H)=\{1\}$.
Since $\pi_\gen\circ\Phi$ is surjective, $(\pi_\gen\circ\Phi)(H)$ is a narrow normal subgroup of $\Aut(W_\gen)$, hence it suffices to show that $\Aut(W_\gen)$ contains no non-trivial narrow normal subgroup.

Let $K$ be a narrow normal subgroup of $\Aut(W_\gen)$.
Let $W_{X_1},\dots,W_{X_p}$ be the generic components of $W$.
By Theorem \ref{thm4_4} there exists a homomorphism $\beta:\Aut(W_\gen)\to\SSS_p$ such that, for all $f\in\Aut(W_\gen)$ and all $i\in\{1,\dots,p\}$, we have $f(W_{X_{i}})=W_{X_{\beta(f)(i)}}$.
Note that
\[
\Ker(\beta)=\Aut(W_{X_{1}})\times\cdots\times\Aut(W_{X_p})\,.
\]
For each $i\in\{1,\dots,p\}$ the projection of $K\cap\Ker(\beta)$ onto $\Aut(W_{X_i})$ is a narrow normal subgroup of $\Aut(W_{X_i})$, hence, by Proposition \ref{prop4_3}, this subgroup is trivial.
We deduce that $K\cap\Ker(\beta)=\{1\}$.
Suppose $K\neq\{1\}$.
Let $h\in K\setminus\{1\}$.
We have $\beta(h)\neq1$ since $K\cap\Ker(\beta)=\{1\}$.
So, there exist $i,j\in\{1,\dots,p\}$, $i\neq j$, such that $h(W_{X_i})=W_{X_j}$.
Let $w\in W_{X_i}\setminus\{1\}$.
We have $\gamma_wh^{-1}\gamma_w^{-1}\in K$, since $K$ is normal, hence
\[
h(\gamma_wh^{-1}\gamma_w^{-1})=(h\gamma_wh^{-1})\gamma_w^{-1}=\gamma_{h(w)}\gamma_{w^{-1}}\in K\cap
\Ker(\beta)=\{1\}\,.
\]
It follows that $\gamma_{h(w)}=\gamma_w$, which is not possible because $\gamma_w\in\Aut(W_{X_i})$, $\gamma_{h(w)}\in\Aut(W_{X_j})$, and these two automorphisms are non-trivial.
So, $K=\{1\}$.

The proof of Part (2) is similar to that of Part (1).
Let $H$ be a normal subgroup of $\Aut(W)$.
We know from Lemma \ref{lem4_6} that $\Ker(\Phi)$ is finite, hence, if $H$ is a subgroup of $\Ker(\Phi)$, then $H$ is finite.

Suppose $H$ is finite.
We want to show that $H\subset\Ker(\Phi)$, that is, that $\Phi(H)=\{1\}$.
Since $\Phi$ is surjective, $\Phi(H)$ is a finite normal subgroup of $\Aut(W_\gen)\times\Aut(W_\aff)$, hence it suffices to show that $\Aut(W_\gen)\times\Aut(W_\aff)$ contains no non-trivial finite normal subgroup.

Let $K$ be a finite normal subgroup of $\Aut(W_\gen)\times\Aut(W_\aff)=\Aut(W_\gen\times W_\aff)$.
Let $W_{X_1},\dots,W_{X_p}$ be the generic components of $W$ and let $W_{X_{p+1}},\dots,W_{X_q}$ be the affine components of $W$.
By Theorem \ref{thm4_4} there exist homomorphisms $\beta:\Aut(W_\gen\times W_\aff)\to\SSS_p$ and $\alpha:\Aut(W_\gen\times W_\aff)\to\SSS_{q-p}$ such that, for all $f\in\Aut(W_\gen\times W_\aff)$, all $i\in\{1,\dots,p\}$ and all $j\in\{1,\dots,q-p\}$, we have $f(W_{X_{i}})=W_{X_{\beta(f)(i)}}$ and $f(W_{X_{p+j}})=W_{X_{p+\alpha(f)(j)}}$.
Note that 
\[
\Ker(\beta)\cap\Ker(\alpha)=\Aut(W_{X_{1}})\times\cdots\times\Aut(W_{X_p})\times\Aut(W_{X_{p+1}})\times
\cdots\times\Aut(W_{X_q})\,.
\]
For each $i\in\{1,\dots,p\}$ the projection of $K\cap\Ker(\beta)\cap\Ker(\alpha)$ onto $\Aut(W_{X_i})$ is a finite normal subgroup of $\Aut(W_{X_i})$, hence, by Proposition \ref{prop4_3}, this subgroup is trivial.
Similarly, for each $j\in\{1,\dots,q-p\}$ the projection of $K\cap\Ker(\beta)\cap\Ker(\alpha)$ onto $\Aut(W_{X_{p+j}})$ is a finite normal subgroup of $\Aut(W_{X_{p+j}})$, hence, by Proposition \ref{prop4_2}, this subgroup is trivial.
So, $K\cap\Ker(\beta)\cap\Ker(\alpha)=\{1\}$.
Suppose $K\neq\{1\}$.
Let $h\in K\setminus\{1\}$.
We have $\beta(h)\neq 1$ or $\alpha(h)\neq 1$, since $K\cap\Ker(\beta)\cap\Ker(\alpha)=\{1\}$.
We assume that $\beta(h)\neq 1$.
The case $\alpha(h)\neq 1$ can be treated in the same way.
There exist $i,j\in\{1,\dots,p\}$, $i\neq j$, such that $h(W_{X_i})=W_{X_j}$.
Let $w\in W_{X_i}\setminus\{1\}$.
We have $\gamma_wh^{-1}\gamma_w^{-1}\in K$ since $K$ is normal, hence
\[
h(\gamma_wh^{-1}\gamma_w^{-1})=(h\gamma_wh^{-1})\gamma_w^{-1}=\gamma_{h(w)}\gamma_{w}^{-1}\in K\cap
\Ker(\beta)\cap\Ker(\alpha)=\{1\}\,.
\]
It follows that $\gamma_{h(w)}=\gamma_w$, which is not possible because $\gamma_w\in\Aut(W_{X_i})$, $\gamma_{h(w)}\in\Aut(W_{X_j})$ and both automorphisms are non-trivial.
So, $K=\{1\}$.
\end{proof}

\begin{corl}\label{corl4_9}
Let $W$ be a Coxeter group and let $H$ be a normal subgroup of $\Aut(W)$.
\begin{itemize}
\item[(1)]
Assume that $W$ has no spherical component, that is, $W_\sph=\{1\}$.
Then $\Aut(W)$ has no non-trivial finite normal subgroup, and $H$ is narrow if and only if $H$ is a subgroup of $\Aut(W_\aff)$.
\item[(2)]
Assume that $W$ has no spherical component and no affine component, that is, $W_\sph=W_\aff=\{1\}$.
Then $\Aut(W)$  has no non-trivial narrow normal subgroup.
\end{itemize}
\end{corl}


\section{Automatic continuity for Coxeter groups and their automorphism groups}\label{sec5}

In this chapter we provide an application of Corollary \ref{corl3_5}\,(i) and Corollary \ref{corl4_9}\,(i) in the direction of automatic continuity of Coxeter groups and their automorphism groups.

Given a map $\varphi\colon L\to G$ between a locally compact Hausdorff group  $L$ and a discrete group $G$, the \emph{automatic continuity problem} is the following: assuming $\varphi$ is a
group homomorphism on the level of groups, find conditions on $G$ or $\varphi$ which imply
that $\varphi$ is continuous.

A discrete group $G$ is called \emph{lcH-slender} if every group homomorphism $\varphi\colon L\to G$ on the level of groups (i.e. abstract group homomorphism) where $L$ is a locally compact Hausdorff group is continuous. 
Many groups are known to be lcH-slender, for example free and free abelian groups (see Dudley \cite{Dudley}).

We call a discrete group $G$ \emph{almost lcH-slender} if every abstract surjective group homomorphism $\varphi\colon L\to G$ from  a locally compact Hausdorff group $L$ onto $G$ is continuous.
An algebraic description of almost lcH-slender groups was proven in \cite[Theorem B]{KeppelerMoellerVarghese}. 
Here we recall a weaker version of this theorem which is suitable for Coxeter groups and their automorphism groups.

\begin{thm}[Keppeler--M\"oller--Varghese \cite{KeppelerMoellerVarghese}]\label{thm5_1}
Let $\varphi\colon L\twoheadrightarrow G$ be an epimorphism from a locally compact Hausdorff group $L$ to a countable group $G$. 
If every torsion subgroup of $G$ is finite, $G$ does not contain $\mathbb{Q}$ and $G$ does not have non-trivial finite normal subgroups, then $\varphi$ is continuous.
\end{thm}

Combining Corollary \ref{corl3_5}\,(i) with Theorem \ref{thm5_1} we obtain a characterization of almost lcH-slender Coxeter groups.

\begin{corl}
Let $W$ be a Coxeter group. We decompose $W$ as a direct product $W_\gen\times W_\aff\times W_\sph$.
The group $W$ is almost lcH-slender if and only if $W_\sph=\left\{1\right\}$.
\end{corl}

\begin{proof}
If $W_\sph$ is non-trivial, then there exists always a discontinuous surjective group homomorphism from the compact group $L:=\prod_{\mathbb{N}}W_\sph$ into $W_\sph$ (see \cite[Example 4.2.12]{Ribes}). 
We denote this epimorphism by $\varphi_\sph$. 
Thus the group homomorphism $$\varphi\colon W_\gen\times W_\aff\times L\to W_\gen\times W_\aff\times W_\sph$$ where $\varphi$ is identity on $W_\gen\times W_\aff$, and $\varphi_{|L}:=\varphi_\sph$ is a surjective discontinuous group homomorphism.

For the other direction, since $W$ is a finitely generated linear group we know that torsion subgroups in $W$ are finite. 
Further, since $W$ is residually finite, $W$ cannot have $\mathbb{Q}$ as a subgroup. 
Thus, if $W_\sph=\left\{1\right\}$ Corollary \ref{corl3_5}\,(i) implies that $W$ does not have non-trivial finite normal subgroups. 
Hence by Theorem \ref{thm5_1} we know that $W$ is almost lcH-slender.
\end{proof}

To show similar results for automorphism groups of Coxeter groups we need the knowledge that torsion subgroups in these automorphism groups are always finite.

Recall, a group $G$ is called \emph{residually $p$-finite} for a prime number $p$ if for every $g\in G$, $g\neq 1$ there exists an epimorphism $\varphi_p\colon G\to F$ where $F$ is a finite $p$-group and $\varphi_p(g)\neq 1$. A direct consequence of residually $p$-finiteness is that every finite order element in a residually  $p$-finite group has $p$-power order. Thus, if $G$ is residually $p$-finite and residually $q$-finite for $p\neq q$, then $G$ is torsion-free.
Further, a group $G$ is called \emph{virtually residually $p$-finite} if there exists a subgroup $H\subseteq G$ of finite index which is residually $p$-finite.

\bigskip\noindent
{\bf Proposition 5.3.}
{\it Let $W$ be a Coxeter group. Then $\Aut(W)$ is virtually torsion-free, in particular every torsion subgroup of $\Aut(W)$ is finite.}

\bigskip\noindent
{\bf Proof.}
It was proven in \cite[p. 91]{Bourb1} that $W$ is linear. Thus, by Platonov's theorem \cite{Platonov} we know that there exist prime numbers $p, q$, $p\neq q$ such that $W$ is virtually residually $p$- and $q$-finite. Further, it was proven in \cite[Proposition 2]{Lubotzky} that the automorphism group of a finitely generated virtually residually $p$-finite group is also virtually residually $p$-finite. Using this fact we know that $\Aut(W)$ is virtually $p$- and $q$-finite. Hence $\Aut(W)$ is virtually torsion-free and therefore every torsion subgroup in $\Aut(W)$ is finite.

\qed

\bigskip\noindent
{\bf Corollary 5.4}
{\it Let $W$ be a Coxeter group. We decompose $W$ as a direct product $W_\gen\times W_\aff\times W_\sph$.
	If $W_\sph=\left\{1\right\}$, then $\Aut(W)$ is almost lcH-slender.}

\bigskip\noindent
{\bf Proof.}
It was proven in \cite{Baumslag} that the automorphism group of a finitely generated residually finite group is itself residually finite. Thus $\Aut(W)$ is residually finite and therefore can not contain $\mathbb{Q}$ as a subgroup.
By assumption, the spherical part of $W$ is trivial, thus Corollary 4.9(i) implies that $\Aut(W)$ has no non-trivial finite normal subgroups. Combining these results with Proposition 5.3, all the assumptions on $\Aut(W)$ which needed in Theorem 5.1 are satisfied, hence $\Aut(W)$ is almost lcH-slender. 

\qed


\end{document}